\documentclass{elsarticle}

\usepackage{amsmath,amsthm,amsopn,amstext,amscd,amsfonts,amssymb}
\usepackage{dsfont} 
\usepackage{comment}
\usepackage[active]{srcltx}
\usepackage{graphicx}
\usepackage{appendix}
\usepackage{lipsum}
\usepackage{mathtools, nccmath}
\usepackage{booktabs}
\usepackage{pdflscape}
\usepackage{multicol}
\usepackage{colortbl}

\usepackage{hyperref}

\DeclareMathOperator{\csch}{csch}
\DeclareMathOperator{\sech}{ sech} 

\usepackage[T2A]{fontenc}
\usepackage[utf8]{inputenc}
\usepackage[russian,english]{babel}

\newtheorem{definition}{\sc Definition}[section]
\newtheorem{theorem}{\sc Theorem}[section]

\newtheorem{proposition}{\sc Proposition}[section]
\newtheorem{lemma}{\sc Lemma }[section]
\newtheorem{eje}{\sc Example }[section]
\newtheorem{coro}{\sc Corollary}[section]
\newtheorem{obs}{\sc Remark}[section]

\numberwithin{figure}{section}
\numberwithin{table}{section}

\usepackage[active]{srcltx}
\def\downbar#1{
\setbox10=\hbox{$#1$}
            \dimen10=\ht10 \advance\dimen10 by 2.5pt
            \ifdim \dimen10<15pt 
               \advance\dimen10 by -0.5pt
               \dimen11=\dimen10
               \advance\dimen10 by 2.5pt
               \lower \dimen11
            \else \lower \ht10 \fi
            \hbox {\hskip 1.5pt \vrule height \dimen10 depth \dp10}}
\def\upbar#1{
\setbox10=\hbox{$#1$}
            \dimen10=\ht10 \advance\dimen10 by \dp10 \advance\dimen10 by 2.5pt
            \ifdim \dimen10<15pt 
                \advance\dimen10 by 2pt \fi
            \raise 2.5pt \hbox {\hskip -1.5pt \vrule height \dimen10}}

\newcommand{\dps}{\displaystyle}

\makeatletter\renewcommand\theenumi{\@roman\c@enumi}\makeatother

\newmuskip\pFqskip
\mathchardef\pFcomma=\mathcode`, 

\newcommand*\pFq[5]{%
 \begingroup
 \begingroup\lccode`~=`,
   \lowercase{\endgroup\def~}{\pFcomma\mkern\pFqskip}%
 \mathcode`,=\string"8000
 {}_{#1}F_{#2}\left(\left.\genfrac..{0pt}{}{#3}{#4}\right|\,#5\right)%
 \endgroup
}

\newcommand*\pPq[5]{%
 \begingroup
 \begingroup\lccode`~=`,
   \lowercase{\endgroup\def~}{\pFcomma\mkern\pFqskip}%
 \mathcode`,=\string"8000
 {}_{#1}\phi_{#2}\left(\left.\genfrac..{0pt}{}{#3}{#4}\right|\,#5\right)%
 \endgroup
}

\begin{document}
\title{On variation of zeros of classical discrete orthogonal polynomials}
\author{K. Castillo}
\ead{kenier@mat.uc.pt}
\address{CMUC, Department of Mathematics, University of Coimbra, 3001-501 Coimbra, Portugal}

\author{F. R. Rafaeli}
\ead{rafaeli@ufu.br}
\address{FAMAT-UFU, Department of Mathematics, University of Uberl\^andia, 38408-100 Uber{l\^a}ndia, Minas Gerais, Brazil}

\author{A. Suzuki}
\ead{alexandresuzuki@ufu.br}
\address{FAMAT-UFU, Department of Mathematics, University of Uberl\^andia, 38408-100 Uber{l\^a}ndia, Minas Gerais, Brazil}

\date{\today}

\begin{keyword} Classical discrete orthogonal polynomials \sep linear grid \sep quadratic grid \sep q-linear grid \sep q-quadratic grid \sep discrete Stieltjes theorem \sep monotonicity \sep zeros.
\MSC[2010]{30C15, 05A30, 33C45, 33D15}
\end{keyword}

\begin{abstract}
The purpose of this note is to establish, from the hypergeometric-type difference equation introduced by Nikiforov and Uvarov, new tractable sufficient conditions for the monotonicity with respect to a real parameter of zeros of classical discrete orthogonal polynomials. This result allows one to carry out a systematic study of the monotonicity of zeros of classical orthogonal polynomials on linear, quadratic, q-linear, and q-quadratic grids. In particular, we analyze in a simple and unified way the monotonicity of the zeros of Hahn, Charlier, Krawtchouk, Meixner, Racah, dual Hahn, q-Meixner, quantum q-Krawtchouk, q-Krawtchouk, affine q-Krawtchouk, q-Charlier, Al-Salam-Carlitz, q-Hahn, little q-Jacobi, little q-Laguerre/Wall, q-Bessel, q-Racah and dual q-Hahn polynomials.
 \end{abstract}
\maketitle
\section{Introduction}\label{intro}
The properties of {\em Jacobi polynomials}, $P_n^{(\alpha, \beta)}$ $(n=1,2,\dots; \alpha>-1, \beta>-1)$, may well consult in \cite[Chapter IV]{S75}; they are orthogonal on $[-1,1]$ with the weight function $(1-X)^\alpha(1+X)^\beta$, and
$$
y=P_n^{(\alpha, \beta)}(X)=\pFq{2}{1}{-n, n+\alpha+\beta+1}{1+\alpha}{\frac{1-X}{2}}
$$ 
satisfy the homogeneous differential equation of second order
\begin{align}\label{edo}
a \,y''+ b \, y'+c\,y=0,
\end{align}
where $a=a(X)=-X^2+1$ and $b=b(X; \alpha, \beta)= -(\alpha+\beta+2)X-\alpha+\beta$. (Having the explicit expression of $P_n^{(\alpha, \beta)}$, the parameter $c$ does not play any interesting role and it can be easily calculated.) Recall that the hypergeometric function ${}_{i}F_{j}$ is formally defined by the series (see \cite[(2.1.2)]{AAR99})
$$
\pFq{i}{j}{\alpha_1,\dots, \alpha_i}{\beta_1\dots, \beta_j}{X}=\sum_{k=0}^\infty \frac{(\alpha_1,\dots, \alpha_i)_k}{(\beta_1\dots, \beta_j)_k}\frac{X^k}{k!},
$$
where
$$
(\alpha_1,\dots, \alpha_i)_k=(\alpha_1)_k\cdots(\alpha_i)_k, \quad (\alpha_1)_k=\prod_{j=1}^k (\alpha_1+j-1);
$$ 
by convention, the empty product is $1$. From \eqref{edo}, Stieltjes proves an important statement concerning the dependence of the zeros of Jacobi polynomials on the parameters $\alpha$ and $\beta$  (see \cite{S87}). Soon after his work has been accepted for publication, Stieltjes implicitly acknowledges, in a note added at the end of the manuscript itself and also in a letter of February 3, 1887 to Hermite (see \cite[Lettre 106]{SH}), that the monotonicity of the zeros of Jacobi polynomials was previously proved by A. Markov in \cite{M86}. However, in Stieltjes' work, {\em ut in multis aliis rebus}, the ``How” is more important than the ``What” and, as he wrote to Hermite, {\em ``la d$\acute{e}$monstration que j'ai d$\acute{e}$velopp$\acute{e}$e pour les Acta Mathematica est diff$\acute{e}$rente de celle de M. Markoff''}. Historically, Stieltjes proves that given a positive definite real symmetric matrix with non-positive off-diagonal elements\footnote{These matrices are now known as Stieltjes' matrices (see \cite[Definition 3.4]{V62}).}, its inverse is also positive definite.  As a consequence, putting aside a clever manipulation of the differential equation \eqref{edo}, he shows that, since
\begin{align}\label{s1}
\frac ba=\frac{-(\alpha+\beta+2)X-\alpha+\beta}{-X^2+1}
\end{align}
 is a strictly decreasing function of $X\in (-1,1)$, and
\begin{align}\label{s2}
\frac{\partial}{\partial \alpha}\left(\frac ba\right)=\frac{1}{X-1}<0, \quad \frac{\partial}{\partial \beta}\left(\frac ba\right)=\frac{1}{X+1}>0,
\end{align}
for each $X\in (-1,1)$, the zeros of Jacobi polynomials are strictly decreasing functions of $\alpha$ on $(-1, \infty)$ and strictly increasing functions of $\beta$ on $(-1, \infty)$ (see \cite[(6)-(7)]{S87}). The proof of Markov is entirely different from that of Stieltjes and is based on the weight function. While it is true that Stieltjes worked directly with Jacobi polynomials ---and Markov proves his result through a general theorem---, his argument furnishes similar results for a more general differential equation (see \cite[Section $6.22$]{S75}). Stieltjes himself considered the ultraspherical case $\alpha=\beta$ and Szeg\H{o} noted that {\em ``the same method applies to Laguerre polynomials''} (see \cite[p. 123]{S75}). In particular, by definition, the {\em old classical orthogonal polynomials on the real line} (Hermite, Jacobi, and Laguerre)\footnote{We write ``classical orthogonal polynomials on the real line" rather than simply  ``classical orthogonal polynomials" because, for instance, from the algebraic point of view of Maroni (see \cite{M91}), the Jacobi polynomials exist and are ``classical'' even when $-\alpha, -\beta, -\alpha-\beta+1\not\in \mathbb{N}$  (see \cite[Chapters $8$ and $9$]{P18} for a recent survey on the subject). Moreover, the Bessel polynomials are classical in the same sense as the other three systems. As Maroni says, ``comme dans le roman d'Alexandre Dumas, les trois mousquetaires \'etaient quatre en r\'ealité''.} are solutions of a differential equation of the same type of \eqref{edo} (see \cite[Section 4.2]{Si15O}). However, only Jacobi and Laguerre polynomials depend on a real parameter and, therefore, the Stieltjes result is no longer applicable in this framework.

From a truly practical point of view ---for example in Physics, which was traditionally the birthplace of some of the most beautiful families (see \cite{VK91})---, it rarely will require more than the classical orthogonal polynomials. For classical (discrete) orthogonal polynomials on a uniform grid \footnote{A non-empty (totally) ordered set of equidistant (respectively, non-equidistant) points is called uniform (respectively, non-uniform)  grid, which in turn is an elementary example of lattice.} (Charlier, Krawtchouk, Hahn, and Meixner polynomials), Markov's theorem can be used (see \cite[Chapter $7$]{I05}) and, of course, for other discrete families \footnote{For a ``continuous'' case as the Askey-Wilson polynomials, Askey and Wilson used a consequence of Markov's theorem, which goes back to Szeg\H{o} (see \cite[Theorem 6.12.2]{S75}), to study the monotonicity of zeros of these polynomials (see \cite[Section $7$]{AW85}).}. Here we purpose an alternative approach, establishing a bridge between two works separated in time by almost one century, on one hand the work of Stieltjes and, on the other hand, the work \cite{NU83} by Nikiforov and Uvarov. This allows one to carry out a systematic study of the monotonicity of zeros of classical orthogonal polynomials on linear, quadratic, q-linear, and q-quadratic grids. Indeed, the purpose of this note is to prove that under suitable regularity conditions the Stieltjes result for Jacobi polynomials remains valid if we replace the differential equation \eqref{edo}, with $a$ and $b$ arbitrary polynomials of degree at most $2$ and $1$, respectively, and $c$ an arbitrary constant (from now on, whenever we refer to \eqref{edo}, we are assuming these conditions), by the following difference equation introduced in \cite[(5)]{NU83} (see also \cite[p. $127$]{NU84} and \cite[p. $71$]{NSU85}):
\begin{align}\label{edd}
a(X)\frac{\Delta}{\Delta x\left(s-1/2\right)}
\left(\frac{\nabla y(X)}{\nabla X}\right)
+\frac{b(X)}{2}\left(\frac{\Delta y(X)}{\Delta X}+\frac{\nabla y(X)}{\nabla X}\right)+c\,y(X)=0,
\end{align}
or, equivalently,
\begin{align}\label{edd2}
a(s)\frac{\Delta}{\Delta x\left(s-1/2\right)}
\left(\frac{\nabla y(X)}{\nabla X}\right)
+b(X)
\frac{\Delta y(X)}{\Delta X}+c\, y(X)=0,
\end{align}
where 
\begin{align*}
a(s)=a(X)-\frac{1}{2}\,b(X)\Delta x\left(s-\frac12\right),
\end{align*}
$X=x(s)$\footnote{In first edition of ``{S}pecial {F}unctions of {M}athematical {P}hysics'' (see \cite{NU78}) the author only consider the case $x(s)=s$. The second edition (see \cite{NU84}) was significantly enriched with the equation \eqref{edd}; although A. A. Samarskii's preface is the same in both editions.}  defines a class of grids with, generally nonuniform, step-size $\Delta X=\Delta x(s)=x(s+1)-x(s)$ and $\nabla X=\nabla x(s)=\Delta x(s-1)$. (By abuse of notation, we use the same letter $a$ for the function $a(s)$ and the polynomial $a(X)$.) In what follows, we assume that $x$ is a real-valued function defined on an interval of the real line. Any solution of \eqref{edd} can be brought in correspondence with the solution of \eqref{edo} by replacing $s$ by $s/h$ and then taking limit $h \to 0$, whenever it exists. It is important to highlight that \eqref{edd} has polynomial solutions, in $X$, whose difference-derivatives satisfy equations of the same kind if and only if, for $q\neq1$ fixed, $x$ is a linear, quadratic, q-linear, or q-quadratic grid of the form
 \begin{align*}
x(s)=
\left\{
\begin{array}{ll}
C_1s^2+C_2s,\\[7pt]
C_3q^{-s}+C_4q^s,
\end{array}
\right.
\end{align*}
where $(C_1,C_2)\neq(0,0)$ and $(C_3,C_4)\neq(0,0)$ (see \cite[(1.68)]{ARS95}). The grids that depend on ``q'' are called q-linear if $C_3$ or $C_4$ is zero; otherwise it is q-quadratic. By using linear transformations (see \cite[(3.4.1)]{NSU}) we can reduce the expressions for the grids to simpler forms. In what follows, we assume that the grid $x$ takes on the following canonical forms (see \cite[p. 74]{NSU}): 
 \begin{align}\label{grid}
x(s)=
\left\{
\begin{array}{lll}
s & \qquad \qquad (\mathrm{I})\\[7pt]
s(s+1) & \qquad \qquad (\mathrm{II})\\[7pt]
q^{s} \qquad \quad \quad \ \ \ \ \ \ \ (q>1) & \qquad \qquad (\mathrm{III})\\[7pt]
\dps\frac12(q^s-q^{-s}) \qquad (q>1) & \qquad \qquad (\mathrm{IV}) \\[7pt]
\dps\frac12(q^s+q^{-s}) \qquad (q>1) & \qquad \qquad (\mathrm{V})\\[7pt]
\dps\frac12(q^s+q^{-s}) \qquad \, (q=e^{2i\theta},\, 0<\theta<\pi/2). & \qquad \qquad (\mathrm{VI})
\end{array}
\right.
\end{align}
Under the above notation, we adopt the following definition of classical discrete orthogonal polynomials on the real line (COPRL), which is enough for our purpose:

\begin{definition}\label{def1}
Fix $\mathrm{a}\in\mathbb{R}\cup\{-\infty\}$ and $N\in\mathbb{N}\cup \{\infty\}$ and define $\mathrm{b}=\mathrm{a}+N$. Fix $q$ and let $X=x(s)$ be a real-valued function given by \eqref{grid}, where the variable $s$ ranges over the finite interval $[\mathrm{a}, \mathrm{b}]$ or the infinity interval $[\mathrm{a}, \infty)$. A sequence of polynomials, $(P_n(X))_{n=0}^{N-1}$,  is said to be sequence of classical discrete orthogonal polynomials on the set $\{x(\mathrm{a}), x(\mathrm{a+1}), \dots, x(\mathrm{b-1})\}$ or, simply, COPRL if $:$\\
\begin{enumerate}
\item  $P_n$ satisfy \eqref{edd}, $x$ being a strictly  monotone function on $[\mathrm{a}, \mathrm{b}]$ or $[\mathrm{a}, \infty)$ given, up to a linear transformation, by \eqref{grid};\\[1pt]
\item   a positive weight function $\omega$ satisfying the boundary conditions \footnote{If $\mathrm{a}$ and $\mathrm{b}$ are finite, \eqref{bc} can be written in the form $\omega(\mathrm{a})a(\mathrm{a})=0$ and $\omega(\mathrm{b})a(\mathrm{b})=0$. If $\mathrm{a}=-\infty$ and/or $\mathrm{b}=\infty$, then \eqref{bc} must hold for each $k$ in the limiting sense.}
\begin{align}\label{bc}
\omega(s) a(s) x^k\left.\left(s-\frac12\right)\right|_{\mathrm{a}, \mathrm{b}}=0 \quad (k=0,1,\dots)
\end{align}
exists;
\item the difference equation 
\begin{align}\label{pearson}
\frac{\Delta}{\Delta x \left(s-\dps\frac12 \right)}\big(\omega(s)a(s)\big)=\omega(s)b(X)
\end{align}
holds.
\end{enumerate}
\end{definition}

Subsequently, when we say that a certain polynomial is a COPRL, we are assuming the definition and notation given in Definition \ref{def1}. From \eqref{edd2}, \eqref{bc}, and \eqref{pearson}, we conclude that the COPRL satisfy the orthogonality condition  (see \cite[(3.3.4)]{NSU})
\begin{align}\label{orth}
\sum_{s=\mathrm{a}}^{\mathrm{b}-1} P_m(X)P_n(X)\omega(s)\Delta x\left(s-\frac12\right)=0 \quad (m\not=n). 
\end{align}
The polynomials $(P_n(X))_{n=0}^{N-1}$ given in Definition \ref{def1} are called simply discrete orthogonal polynomials, in $X$, on the set $\{x(\mathrm{a}), x(\mathrm{a+1}), \dots, x(\mathrm{b-1})\}$ with respect to a positive weight function $\omega$ if they satisfy the relation \eqref{orth} instead of the requirements $\mathrm{i})-\mathrm{iii})$. Of course, COPRL are a special case of discrete orthogonal polynomials. 
From \eqref{orth}, we can see that the zeros of discrete orthogonal polynomials on $\{x(\mathrm{a}), x(\mathrm{a+1}), \dots, x(\mathrm{b-1})\}$ are real and distinct and are located in $(\min\{x(\mathrm{a}), x(\mathrm{b-1})\}, \max\{x(\mathrm{a}), x(\mathrm{b-1})\})$ (see \cite[Theorem 3.3.1]{S75}). In concluding this section we remark that it is possible to obtain a series representation of COPRL (see \cite[(4.19)]{ARS95} and \cite[Section 3]{NU93}):
\begin{align}\label{explicit}
&P_n(X)=(-1)^n \gamma_n \sum_{j=0}^n \dps\binom{n}{j}_q \big(-q^{-n/2}\big)^j\frac{\Delta x\big(s-(n-1)/2+j\big)}{\dps \prod_{k=0}^n \Delta x\big(s+(j-k+1)/2\big)}\\[7pt]
\nonumber&\times  \prod_{k=0}^n a(s-n+j+k)\prod_{l=0}^j \frac{a(x(s+l-1))+1/2\, b(x(s+l-1))\Delta x(s+l-1/2)}{a(x(s+l))-1/2\, b(x(s+l)) \Delta x(s+l+1/2)},
\end{align}
where $\gamma_n$ is a constant and
$$
\binom{n}{j}_q=\frac{(q; q)_n}{(q; q)_j (q; q)_{n-j}}.
$$
Furthermore (see \cite{NU93}), COPRL represent special cases of hypergeometric series or q-hypergeometric series, the latter formally defined by (see \cite[(10.9.4)]{AAR99})
$$
\pPq{i}{j}{\alpha_1,\dots, \alpha_i}{\beta_1\dots, \beta_j}{q, X}=\sum_{k=0}^\infty \frac{(\alpha_1,\dots, \alpha_i; q)_k}{(\beta_1\dots, \beta_j; q)_k}(-1)^{(1-i+j)k\dps\binom{k}{2}}\frac{X^k}{(q; q)_k}.
$$

The outline of this note is as follows. In Section \ref{ext} we present an extension of the Stieltjes work in the framework of COPRL. In Section \ref{app} we study the variation of zeros of some families of COPRL according to the type of underlying grid.
\section{Main results}\label{ext}

Unless otherwise stated we assume that $a$, $b$, and $c$ appearing in \eqref{edd} depend on a parameter $t$ varying in a non-degenerate open interval of the real line. Rewrite  \eqref{edd} in the more suggestive form
\begin{align}\label{DE2}
A\, y(s-1)+B\, y(s+1)+C\, y(s)=0,
\end{align}
where $y(s)=y(X; t)$ and 
\begin{align}
\label{AB} A=A(s; t)&=\frac{a(s; t)}{\nabla X\Delta x(s-1/2)}, \quad B=B(s; t)=\frac{a(s; t)+b(X; t)\Delta x(s-1/2)}{\Delta X\Delta x(s-1/2)},\\[7pt]
\nonumber C=C(s; t)&=c(t)-A(s; t)-B(s; t).
\end{align} 
The following example will help motivate our main result.

\begin{eje}\label{eje1}
In 1960, Karlin and McGregor proved  (see \cite[(1.3)]{KM60} and \cite[(5.1)]{L67}) that the COPRL known as Hahn polynomials (see \cite[Section 9.5]{KLS10})
$$
y(s)=H^{(\alpha, \beta)}_n(X)=\pFq{3}{2}{-n, -X, \alpha+\beta+n+1}{\beta+1, 1-N}{1}
$$
$(n=1,\dots, N-1; \mathrm{a}=0, \mathrm{b}=N; \alpha>-1, \beta>-1)$, ---which constitute the finite discrete analogue of Jacobi polynomials considered by Stieltjes in \cite{S87}--- satisfy \eqref{DE2} with $X=x(s)=s$, and $A$ and $B$ given by
\begin{align*}
A=A(X; \alpha)=X(-X+\alpha+N),\quad B=B(X; \beta)=(X+\beta+1)(-X+N-1).
\end{align*}
The function
\begin{align}\label{S1}
\frac BA=\frac{(X+\beta+1)(-X+N-1)}{X(-X+\alpha+N)}
\end{align}
 is a positive and strictly decreasing function of $X \in(\mathrm{a}, \mathrm{b}-1)$, and
\begin{align}
\label{S21}\frac{\partial}{\partial \alpha}\left(\frac BA\right)&=\frac{(X+\beta+1)(X-N+1)}{X(-X+\alpha+N)^2}<0,\\[7pt]
\label{S22} \frac{\partial}{\partial \beta}\left(\frac BA\right)&=\frac{-X+N-1}{X(-X+\alpha+N)}>0.
\end{align}
for each $X \in(\mathrm{a}, \mathrm{b}-1)$. Since this example corresponds to the  grid $(\mathrm{I})$, from Markov's theorem, Ismail proves that the zeros of $H^{(\alpha, \beta)}_n$ are decreasing functions of $\alpha$ on $(-1, \infty)$ and increasing functions of $\beta$ on $(-1, \infty)$ (see \cite[Theorem $7.1.2$]{I05}).  Comparison of \eqref{s1} and \eqref{S1}, and \eqref{s2} and \eqref{S21}-\eqref{S22} suggests that, as for Jacobi polynomials, the information on the monotonicity of the zeros of Hahn polynomials is stored in the rational function $B/A$. Therefore, it is not unreasonable to conjecture that the same happens with any COPRL.
\end{eje}

The following two mathematical objects play a central role in our exposition.
 
\begin{definition}
Let $A$ and $B$ be given by \eqref{AB}. The function $f$, from now on called monotonicity function, is defined by
\begin{align}\label{f}
f(s; t)=\frac{B(s; t)}{A(s; t)}=\frac{a(s;t)+b(X; t)\Delta x(s-1/2)}{a(s; t)}\frac{\nabla X}{\Delta X}.
\end{align}
\end{definition}

\begin{definition}
Let $X=x(s)$ be given by \eqref{grid} and let $x(y_j(t))$ $(j=1,\dots,n)$ be the zeros of a COPRL of degree $n$, say $P_n(X; t)$, depending on a parameter $t$ taking values on $\mathrm{J}\subseteq \mathbb{R}$. For any set $\mathrm{I}\subseteq \mathbb{R}$, the (nonempty) subset $\mathrm{S}_\mathrm{I}^{(t)}(P)$ of $(\mathrm{a}, \mathrm{b}-1)$ is defined by
$$
\mathrm{S}_\mathrm{I}^{(t)}(P_n)=\Big\{y\in (\mathrm{a}, \mathrm{b}-1)\,\big|\, t\in \mathrm{I}\cap \mathrm{J}\,\wedge\, \big(\forall j\in \{1,\dots,n \}\big)\big[y=y_j(t)\big]\Big\}.
$$
We will write it simply $\mathrm{S}_\mathrm{I}(P_n)$, $\mathrm{S}_\mathrm{I}^{(t)}$ or $\mathrm{S}_\mathrm{I}$ when no confusion can arise\footnote{Note that $\mathrm{S}_\mathrm{J}=\mathrm{S}_\mathbb{R}$ and  $\mathrm{S}_\mathrm{I}\subseteqq\mathrm{S}_\mathrm{J}$ whenever $\mathrm{I}\subseteqq \mathrm{J}$.}.
\end{definition}  

To prove our main result we need two lemmas. The first one was proved for the Hahn polynomials by Levit  (see \cite[Theorem $3$]{L67}). Here we reproduce in a more general framework, {\em mutatis mutandis}, his arguments.

\begin{lemma}\label{lemma}
Let $X=x(s)$ be given by \eqref{grid}. Let $x(y(t))$ and $x(z(t))$ be consecutive zeros of a COPRL depending on a real parameter $t$. Let $f$ be the monotonicity function given by \eqref{f}. Then $|z(t)-y(t)|>1$ for those values of $t$ such that $f(\cdot; t)>0$ on $\mathrm{S}_{\mathbb{R}}$.
\end{lemma}
\begin{proof}
We give the proof only for the case in which $x$ is a strictly increasing function. The same arguments apply to the case in which $x$ is a strictly decreasing function. There is no loss of generality in assuming that $x(y(t))<x(z(t))$ are the greatest pair of consecutive zeros of a COPRL, say $P(X; t)$. (Recall that the zeros of $P$ are distinct and are located in $(x(\mathrm{a}), x(\mathrm{b-1}))$. Moreover, under our assumptions, $y(t)<z(t)$.) Suppose that there exists $t_0$ such that $z(t_0)=y(t_0)+1$. Replacing $s$ by $z(t_0)$ and $t$ by $t_0$ in (\ref{DE2}) we have $P(x(z(t_0)+1);t_0)=0$, because $f(z(t_0); t_0)\neq0$, which contradicts the assumption that $x(z(t_0))$ is the greatest zero of $P(\cdot; t_0)$. (Here we have also used the fact that $x$ is, in particular, a strictly increasing function on $[\mathrm{a}, \mathrm{b})$, and not only on $[\mathrm{a}, \mathrm{b}-1]$, because the possibility that $z(t_0)+1\in [\mathrm{b}-1,\mathrm{b})$ is not excluded.) Suppose now that $z(t_0)<y(t_0)+1$. Replacing again $s$ by $z(t_0)$ and $t$ by $t_0$ in (\ref{DE2}) we have
$$
-\frac{P(x(z(t_0)-1); t_0)}{P(x(z(t_0)+1); t_0)}=f(z(t_0); t_0)>0.
$$
Since $z(t_0)\not=y(t_0)+1$, there is an odd number of zero of $P(\cdot; t_0)$ in $[x(z(t_0)-1),x(z(t_0)+1)]$. We next claim that there exists at least one integer $m$ $(1\leq m\leq N-1)$ such that $y(t_0)\leq \mathrm{a}+m \leq z(t_0)$ or, what is the same, $x(y(t_0)) \leq x(\mathrm{a}+m) \leq x(z(t_0))$. Suppose the assertion were false. Hence
$$
\sum_{s=\mathrm{a}}^{\mathrm{b}-1} \frac{\dps P^2(X; t_0)}{\big(X-x(y(t_0))\big)\big(X-x(z(t_0))\big)}\omega(s; t_0)\Delta x\left(s-\frac12\right)>0,
$$
a contradiction with \eqref{orth}. Since $z(t_0)<y(t_0)+1$, $m$ is unique. Thus the only zeros of $P(\cdot; t_0)$ in $[x(z(t_0)-1),x(z(t_0)+1)]$ are $x(y(t_0))$ and $x(z(t_0))$. This contradicts the fact that there is an odd number of zeros of $P(\cdot; t_0)$ in $[x(z(t_0)-1),x(z(t_0)+1)]$, and the lemma follows.
\end{proof}

\begin{lemma}\label{Lemma}
Let $X=x(s)$ be given by \eqref{grid}. Let $P$ be a non-constant COPRL, in $X$, satisfying \eqref{DE2} and depending on a parameter $t$ varying in a non-degenerate open interval of the real line containing $t_0$. Suppose that $A(\cdot; t)$ and $B(\cdot; t)$ given by \eqref{AB} admit partial derivatives with respect to $t$ on a neighbourhood of $t_0$. Assume that $P(Y_0; t_0)=0$. Then there exist $\epsilon>0$ and $\delta>0$ such that  $(Y_0-\delta, Y_0+\delta) \times (t_0-\epsilon, t_0+\epsilon)$  is on the neighbourhood where $P$ is defined, and there exists $Y: (t_0-\epsilon, t_0+\epsilon) \to (Y_0-\delta, Y_0+\delta)$, such that
\begin{align}\label{zero}
P(Y(t); t)=0
\end{align}
and, for each $t \in (t_0-\epsilon, t_0+\epsilon)$, $Y$ is the unique solution of \eqref{zero} with $Y(t)\in (Y_0-\delta, Y_0+\delta)$. Moreover, $Y$  possess a continuous derivative on $(t_0-\epsilon, t_0+\epsilon)$.
\end{lemma}
\begin{proof}
From \eqref{explicit} we see that the coefficients of $P(\cdot; t)$ are differentiable functions of $t$. Moreover, $P(Y_0; t_0)=0$; from this it follows that 
$$
\left. \frac{\partial P}{\partial X}(X; t)\right|_{X=Y_0, t=t_0}\not=0,
$$
because the zeros of $P(\cdot; t_0)$ are distinct. Thus, the result is a direct consequence of the implicit function theorem (see \cite[Theorem 3.4.2]{Si15}).
\end{proof}

We shall refer to Theorem \ref{main} below as discrete Stieltjes theorem.

\begin{theorem}\label{main}
Assume the hypotheses and notation of Lemma \ref{Lemma}. Let $f$ be the monotonicity function given by \eqref{f}. Denote $f_1(s; \cdot)=(\partial f/\partial s) (s;\cdot)$ and $f_2(\cdot;t)=(\partial f/\partial t)(\cdot;t)$. Suppose that $f(s;t)>0$ and $f_1(s;t)<0$ on $\mathrm{S}_{(t_0-\epsilon, t_0+\epsilon)}$. In the case of the grid $(\mathrm{IV})$, assume also that $n f(s; \cdot)+f_1(s; \cdot)\leq 0$. Suppose furthermore that $f_2(s;t)>0$ (respectively, $f_2(s;t)<0$) on $\mathrm{S}_{(t_0-\epsilon, t_0+\epsilon)}$. Then $Y$ is a strictly increasing (respectively, decreasing) function on $(t_0-\epsilon, t_0+\epsilon)$ if $x$ is a strictly increasing function, or else $Y$ is a strictly decreasing (respectively, increasing) function on $(t_0-\epsilon, t_0+\epsilon)$.
\end{theorem}
\begin{proof}
We give the proof only for the case in which $x$ is a strictly increasing function. The same arguments apply to the case in which $x$ is a strictly decreasing function. Assume that $P$ is monic and has fixed degree $n$. Let $Y_j$ $(j=1,\dots,n)$ denote the zeros of $P$. Since the indeterminate of $P$ takes values on the real (open) interval $x((\mathrm{a}, \mathrm{b}-1))$, there exist functions $y_j$ defined on a neighbourhood of $t_0$ and taking values on $(\mathrm{a}, \mathrm{b}-1)$ such that $Y_j(t)=x(y_j(t))$. Moreover, since $x$ is strictly  increasing on $(\mathrm{a}, \mathrm{b}-1)$, $x^{-1}$ is differentiable on $x((\mathrm{a}, \mathrm{b}-1))$. Hence, by Lemma \ref{Lemma}, there exists a neighbourhood of $t_0$ where $y_j$ is differentiable. 

Let
\begin{align*}
P(X; t)=\prod_{j=1}^{n}(X-Y_j(t))
\end{align*}
defined in a neighbourhood of $t_0$. Replacing $s$ by $y_j(t)$ in (\ref{DE2}) we have 
\begin{align}\label{eq1}
f(y_j(t); t)=-\frac{P(x(y_j(t)-1); t)}{P(x(y_j(t)+1); t)}.
\end{align}
Taking the partial derivative of \eqref{eq1} with respect to $t$ on the neighbourhood of $t_0$ leads to
\begin{align*}
&y'_j(t)f_1(y_j(t); t)+f_2(y_j(t); t)\\[7pt] 
\nonumber &\quad=\frac{P(x(y_j(t)-1);t)\dps \frac{\partial P}{\partial t}(x(y_j(t)+1); t)-P(x(y_j(t)+1);t)\dps \frac{\partial P}{\partial t}(x(y_j(t)-1); t) }{P^2(x(y_j(t)+1);t)},
\end{align*}
where
\begin{align*}
&\frac{\partial P}{\partial t}(x(y_j(t)\pm1); t)\\[7pt]
\nonumber &\quad=P(x(y_j(t)\pm1); t)\dps \sum_{k=1}^n \frac{\left.\dps\frac{\mathrm{d}x}{\mathrm{d}s}(s\pm1)\right|_{s=y_j(t)} y'_j(t)-  \left.\dps\frac{\mathrm{d}x}{\mathrm{d}s}(s)\right|_{s=y_k(t)} y'_k(t)}{x(y_j(t)\pm1)-Y_k(t)}.
\end{align*}
Hence
\begin{align}\label{f2}
f_{2}(y_j(t); t)=\sum\limits_{k=1}^{n} a_{j k}(t)y'_k(t),
\end{align}
where
\begin{align}
\label{aii}a_{jj}(t)=-f_{1}(y_j(t); t)+f(y_j(t); t)\sum_{k=1}^{n}b_{j k}(t)-\sum\limits_{\substack{k=1 \\ k\neq j}}^{n}a_{jk},
\end{align}
with
\begin{align*}
b_{jk}(t)= \frac{\left.\dps\frac{\mathrm{d}x}{\mathrm{d}s}(s-1)\right|_{s=y_j(t)}-\left.\dps\frac{\mathrm{d}x}{\mathrm{d}s}(s)\right|_{s=y_k(t)}}{x(y_j(t)-1)-Y_k(t)}-\frac{\left.\dps\frac{\mathrm{d}x}{\mathrm{d}s}(s+1)\right|_{s=y_j(t)}-  \left.\dps\frac{\mathrm{d}x}{\mathrm{d}s}(s)\right|_{s=y_k(t)}}{x(y_j(t)+1)-Y_k(t)},
\end{align*}
and $a_{jk}(t)=f(y_j(t); t)\,c_{jk}(t)$ $(j\not=k)$, where
\begin{align*}
c_{jk}(t)=\left(\frac{1}{x(y_j(t)+1)-Y_k(t)}-\frac{1}{x(y_j(t)-1)-Y_k(t)}  \right)\left.\dps\frac{\mathrm{d}x}{\mathrm{d}s}(s)\right|_{s=y_k(t)}.
\end{align*}
We claim that $a_{jk}(t)<0$ $(j\not=k)$. From now on, without loss of generality, we assume $Y_j(t)<Y_k(t)$. Then, by Lemma \ref{lemma}, $x(y_j(t)+1)<Y_k(t)$. Since $\left.(\mathrm{d}x/\mathrm{d}s)(s)\right|_{s=y_k(t)}$ $>0$, $a_{jk}(t)<0$ as required. Consider now the canonical forms \eqref{grid} to conclude that $a_{jj}(t)>0$. We next claim that, for the grids $(\mathrm{I})-(\mathrm{III})$ and $(\mathrm{V})-(\mathrm{VI})$, $b_{jk}(t)>0$. Indeed, for the grids $(\mathrm{I})$ and $(\mathrm{III})$, the proof is straightforward. For the grid $(\mathrm{II})$, we have
$$
b_{jk}(t)=\frac{4}{\big(y_j(t)+y_k(t)\big)\big(y_j(t)+y_k(t)+2\big)}
$$
and $y_j(t)\in \mathrm{S}_{(t_0-\epsilon, t_0+\epsilon)}\subset (-1/2, \infty)$, the latter because $x$ is strictly  increasing on $(-1/2, \infty)$. Hence, by Lemma \ref{lemma}, $b_{jk}(t)>0$ as claimed. Since $f(y_j(t); t)>0$, $f_{1}(y_j(t); t)<0$, $a_{jk}(t)<0$, and $b_{jk}(t)\geq 0$, \eqref{aii} implies $a_{jj}(t)>0$. The corresponding result for the grid $(\mathrm{V})$ follows in the same way after noting that 
$$
b_{jk}(t)=2\theta \sinh (2 \theta) \csch\big((y_k(t)+y_j(t)-1)\theta\big)\csch\big((y_k(t)+y_j(t)+1)\theta \big),
$$
for $q=e^{2\theta}$ ($\theta>0$), and $y_j(t)\in \mathrm{S}_{(t_0-\epsilon, t_0+\epsilon)}\subset (0, \infty)$. And the same goes for  the grid $(\mathrm{VI})$ using hyperbolic identities and noting that $y_j(t)\in \mathrm{S}_{(t_0-\epsilon, t_0+\epsilon)}\subset (-\pi/(2\theta), 0)$. Define 
\begin{align*}
&\quad\mathbf{f}(t)=(f_2(y_1(t); t), \dots, f_2(y_n(t); t))^\mathsf{T},\\[7pt] 
\mathbf{A}&(t)=(a_{jk}(t)),\quad \mathbf{y}(t)=\big(y'_1(t),\dots, y'_n(t) \big)^\mathsf{T},
\end{align*}
and rewrite \eqref{f2} as $\mathbf{f}(t)=\mathbf{A}(t)\mathbf{y}(t)$. Observe that $\mathbf{A}(t)$ has positive diagonal entries and negative of off-diagonal entries. Moreover, from \eqref{aii} we get
$$
 |a_{jj}(t)|>\sum\limits_{\substack{k=1 \\ k\neq j}}^{n}|a_{jk}(t)|.
$$
Hence $\mathbf{A}(t)$ is a real irreducibly diagonally dominant matrix, and so, by \cite[Corollary 1, p. 85]{V62} all the entries of $\mathbf{A}^{-1}(t)$ are positive \footnote{Indeed, $\mathbf{A}(t)$ is a Stieltjes matrix.}. Thus all the entries of $\mathbf{y}(t)=\mathbf{A}^{-1}(t) \mathbf{f}(t)$ are positive, and the theorem is proved for the grids $(\mathrm{I})-(\mathrm{III})$ and $(\mathrm{V})-(\mathrm{VI})$. The above argument does not work for the grid $(\mathrm{IV})$, because, for $q=e^{2\theta}$ ($\theta>0$),
$$
b_{jk}(t)=-2\theta \sinh(2\theta)\sech\big((y_j(t)+y_k(t)+1)\theta\big)\sech\big((y_j(t)+y_k(t)-1)\theta\big)<0.
$$
Indeed, $-1<b_{jk}(t)<0$ and, therefore, under the additional hypothesis for this grid, the theorem follows from \eqref{aii}.
\end{proof}

\begin{obs}
In some cases, the hypotheses of the discrete Stieltjes theorem are fulfilled in $(\mathrm{a}, \mathrm{b}-1)\supset  \mathrm{S}_\mathbb{R}$. However, there are several, and important, examples where this is only true on a subset of $(\mathrm{a}, \mathrm{b}-1)$ containing all the elements of $\mathrm{S}_\mathbb{R}$; see, for instance, the Racah polynomials in Section \ref{latticeII} below. Of course, since the precise location of the zeros of $P$ is not known, these cases require a more careful analysis.
\end{obs}

In Example \ref{eje1}, we can write $H^{(\alpha, \beta, N)}_n$ instead of $H^{(\alpha, \beta)}_n$. The interlacing property between the zeros of $H^{(\alpha, \beta, N)}_n$ and $H^{(\alpha, \beta, N+1)}_n$ was observed by Levit (see \cite[Theorem 6]{L67}). It was proved independently, and almost simultaneously, by Mesztenyi that  \cite[Theorem 6]{L67} is always true for discrete orthogonal polynomials on the grid $(\mathrm{I})$ when the variable $s$ ranges over a finite interval (see \cite[Lemma 3]{Ma66}). In this way, the following general result might be of interest to the reader.

\begin{theorem}\label{N}
Fix $\mathrm{a}\in\mathbb{R}$ and $N\in\mathbb{N}$ and define $\mathrm{b}=\mathrm{a}+N$. Let $X=x(s)$ be a real-valued function, where the variable $s$ ranges over the finite interval $[\mathrm{a}, \mathrm{b}+1]$.  Let $(P_n(X; N))_{n=1}^{N-1}$ be the sequence of discrete orthogonal polynomials, in $X$, on the set $\{x(\mathrm{a}), x(\mathrm{a}+1), \dots,$ $ x(\mathrm{b}-1)\}$. Let $(P_n(X; N+1))_{n=1}^{N}$ be the sequence of discrete orthogonal polynomials, in $X$, on the set $\{x(\mathrm{a}), x(\mathrm{a}+1), \dots, x(\mathrm{b})\}$. Assume that both sequence are orthogonal with respect to the same positive weight function. Let $Y_1<Y_2<\cdots <Y_n$ be the zeros of $P_n(\cdot; N)$. Then one of the following situations holds:
\begin{enumerate}
\item The zeros of $P_n(\cdot; N+1)$ are those of $P_n(\cdot; N)$, if $P_n(x(\mathrm{b}); N)=0$;
\item  $P_n(\cdot; N+1)$ has a zero on $(Y_k, Y_{k+1})$, if $P_n(x(\mathrm{b}); N)\not=0$  and $x(\mathrm{b})\not\in (Y_k, Y_{k+1})$ for fixed $k\in \{1,\dots, n-1\}$;
\item  $P_n(\cdot; N+1)$ has a zero on each of the intervals $(Y_k, x(\mathrm{b}))$ and $(x(\mathrm{b}), Y_{k+1})$, if $P_n(x(\mathrm{b}); N)\not=0$ and $x(\mathrm{b})\in (Y_k, Y_{k+1})$ for fixed $k\in \{1,\dots, n-1\}$.
\end{enumerate}
\end{theorem}
\begin{proof}
To shorten notation, write $P_n^{(N)}$ and $P_n^{(N+1)}$ instead of $P_n(\cdot; N)$ and $P_n(\cdot; N+1)$, respectively. The reader may check for himself that by expressing $P_n^{(N+1)}$ as a linear combination of the elements of the set $\big\{1, P_1^{(N)}, \dots, P_{n}^{(N)}\big\}$ and using the orthogonality property, we obtain $P_n^{(N+1)}(x(\mathrm{b}))=\zeta_n\, P_n^{(N)}(x(\mathrm{b}))$ and
\begin{align}\label{linearcomb}
P_n^{(N+1)}(X)=& P_n^{(N)}(X)\\[7pt]
\nonumber&-\eta_n P_n^{(N+1)}(x(\mathrm{b}))\frac{P_{n-1}^{(N)}(x(\mathrm{b})) P_n^{(N)}(X)-P_n^{(N)}(x(\mathrm{b}))P_{n-1}^{(N)}(X)}{X-x(\mathrm{b})},
\end{align}
where $\eta_n=\omega(\mathrm{b})\Delta x(\mathrm{b}-1/2)\big\|P_{n-1}^{(N)}\big\|^{-2}$ and
\begin{align*}
\zeta_n&=1+\eta_n\left( P_{n-1}^{(N)}(x(\mathrm{b}))  \dps \left.\dps\frac{\mathrm{d}P_n^{(N)}}{\mathrm{d}X}(X)\right|_{X=x(\mathrm{b})}-\dps \left.\dps\frac{\mathrm{d}P_{n-1}^{(N)}}{\mathrm{d}X}(X)\right|_{X=x(\mathrm{b})} P_n^{(N)}(x(\mathrm{b}))\right).
\end{align*}
Following a standard procedure, the rest of the proof follows as \cite[Lemma 3]{Ma66}.
\end{proof}

We end this section with the following consequence of Theorem \ref{N}, which is valid, in particular, for COPRL.

\begin{coro}\label{CN}
Assume the hypotheses and notation of Theorem \ref{N}. Suppose furthermore that $x$ is a strictly monotone function on $[\mathrm{a}, \mathrm{b}+1]$. Set $Y_0=Y_{n+1}=x(\mathrm{b})$. Then $P_n(\cdot; N+1)$ has a zero on each interval $(Y_k, Y_{k+1})$ for all $k\in \{1,\dots, n\}$  if $x$ is a strictly increasing function, or else $P_n(\cdot; N+1)$ has a zero on each interval $(Y_k, Y_{k+1})$ for all $k\in \{0,\dots, n-1\}$.
\end{coro}
\begin{proof}
We give the proof only for the case in which $x$ is a strictly increasing function. The same arguments apply to the case in which $x$ is a strictly decreasing function. From \eqref{linearcomb}, it follows that 
$$
\mathrm{sgn}\left(P_n(Y_n; N+1)\right)=-\mathrm{sgn}\left(P_n(Y_{n+1}; N)\right)=-1.
$$
Hence $P_n(\cdot; N+1)$ has a zero on $(Y_n, Y_{n+1})$. The rest of the proof is a direct consequence of Theorem \ref{N} $\mathrm{ii})$.
\end{proof}

\section{Applications}\label{app}
In this section we apply the discrete Stieltjes theorem to specific families of COPRL. (Of course, Corollary \ref{CN} is applicable to any family of COPRL on a finite grid, see for instance the Hahn, Krawtchouk, Racah, dual Hahn, q-Hahn, q-Krawtchouk, affine q-Krawtchouk, quantum q-Krawtchouk, q-Racah, and dual q-Hanh polynomials below.) We prove, or sketch the proof in similar cases, only of some illustrative examples; the other cases are stated and the proofs are left as exercises for the reader. The reader also should satisfy himself that the hypotheses of Lemma \ref{Lemma} are fulfilled. As far as we know, only the monotonicity of zeros of COPRL on the grid $(\mathrm{I})$ has been studied (see \cite[Chapter $7$]{I05}) \footnote{Recall that  for a continuous case on the grid $(\mathrm{VI})$  as the Askey-Wilson polynomials, the monotonicity of their zeros was studied in \cite[Section 7]{AW85}.}. We include this case for the sake of completeness. 

\subsection{The grid $(\mathrm{I})$}\label{latticeI}
\subsubsection*{Examples of COPRL on $X=x(s)=s$ (Hahn, Charlier, Krawtchouk, and Meixner polynomials).}

The {\em Hahn polynomials}, $H^{(\alpha, \beta)}_n$, are defined in Example \ref{eje1}.
\begin{proposition}
The zeros of $H^{(\alpha, \beta)}_n$ are strictly decreasing functions of $\alpha$ on $(-1,\infty)$ and strictly increasing functions of $\beta$ on $(-1,\infty)$.
\end{proposition}
\begin{proof}
It suffices to use \eqref{S1} and \eqref{S21}-\eqref{S22} together with the discrete Stieltjes theorem.
\end{proof}

The {\em Charlier polynomials} (see \cite[Section 9.14]{KLS10}), 
$$
y(s)=C_n^{(\alpha)}(X)=\pFq{2}{0}{-n, -X}{-}{-\frac{1}{\alpha}}
$$ 
$(n=1,2,\dots; \mathrm{a}=0, \mathrm{b}=\infty; \alpha>0)$, satisfy \eqref{DE2} with $A$ and $B$ given by
\begin{align*}
A=A(X)=X, \quad B=B(\alpha)=\alpha.
\end{align*}

\begin{proposition}
The zeros of $C_n^{(\alpha)}$ are strictly increasing functions of $\alpha$ on $(0, \infty)$.
\end{proposition} 

The {\em Krawtchouk polynomials} (see \cite[Section 9.11]{KLS10}), 
$$
y(s)=K_n^{(\alpha)}(X)=\pFq{2}{1}{-n, -X}{1-N}{\frac{1}{\alpha}}
$$
$(n=1,\dots, N-1; \mathrm{a}=0, \mathrm{b}=N; 0<\alpha<1)$, satisfy \eqref{DE2} with $A$ and $B$ given by
\begin{align*}
A=A(X; \alpha)=(1-\alpha)X,\quad B=B(X; \alpha)=\alpha(-X+N-1).
\end{align*}

\begin{proposition}
The zeros of $K_n^{(\alpha)}$ are strictly increasing functions of $\alpha$ on $(0, 1)$.
\end{proposition}

The {\em Meixner polynomials} (see \cite[Section 9.10]{KLS10}), 
$$
y(s)=M_n^{(\alpha, \beta)}(X)=\pFq{2}{1}{-n, -X}{\beta}{1-\frac{1}{\alpha}}
$$
$(n=1,2,\dots; \mathrm{a}=0, \mathrm{b}=\infty; 0<\alpha<1, \beta>0)$, satisfy \eqref{DE2} with $A$ and $B$ given by
\begin{align*}
A=A(X)=X, \quad B=B(X; \alpha, \beta)=\alpha(X+\beta).
\end{align*}

\begin{proposition}
The zeros of $M_n^{(\alpha, \beta)}$ are strictly increasing functions of $\alpha$ on $(0, 1)$ and strictly increasing functions of $\beta$ on $(0,\infty)$.
\end{proposition}
\subsection{The grid $(\mathrm{II})$}\label{latticeII}
\subsubsection*{Examples of COPRL on $X=x(s)=s(s+1)$ (Racah and dual Hahn polynomials).}

The {\em Racah polynomials} (see \cite[p. 236]{NU93}), 
\begin{align}\label{FRacah}
y(s)=R_n^{(\alpha, \beta)}(X)&=\pFq{4}{3}{-n, \alpha+\beta+n+1, \mathrm{a}-s, s+\mathrm{a}+1}{2\mathrm{a}+\alpha+N+1, \beta+1, 1-N}{1}
\end{align}
$(n=1,\dots, N-1; \mathrm{a}>-1/2, \mathrm{b}=\mathrm{a}+N; \alpha>-1, -1<\beta<2\mathrm{a}+1)$, satisfy \eqref{DE2} with $A$ and $B$ given by
\begin{align*}
A&=A(s; \alpha, \beta)=\frac{(s-\mathrm{a})(s+\mathrm{a}+N)(s-\mathrm{a}-\alpha-N)(s+\mathrm{a}-\beta)}{2s(2s+1)},\\[7pt]
B&=B(s; \alpha, \beta)=\frac{(s+\mathrm{a}+1)(s-\mathrm{a}-N+1)(s+\mathrm{a}+\alpha+N+1)(s-\mathrm{a}+\beta+1)}{2(s+1)(2s+1)}.
\end{align*}

\begin{proposition}\label{Racah}
The zeros of $R_n^{(\alpha,\beta)}$ are strictly decreasing functions of $\alpha$ on $(-1, \infty)$ and strictly increasing function of $\beta$ on $(-1, 2\mathrm{a}+1)$ if $\mathrm{a}\geq 0$, or else the zeros of $R_n^{(\alpha,\beta)}$ are strictly decreasing functions of $\alpha$ on $(-1, \infty)$ for each $\beta\in (\mathrm{a}, 2\mathrm{a}+1)$ and strictly increasing function of $\beta$ on $(\mathrm{a}, 2\mathrm{a}+1)$.
\end{proposition} 
\begin{proof}
We give the proof only for the case in which $\mathrm{a}\geq 0$. The proof for $-1/2<\mathrm{a}<0$ is similar. Define the interval $\mathrm{K}=\big(\max\big\{\mathrm{a}, \beta-\mathrm{a}\big\}, \mathrm{a}+N-1\big)$. 
The monotonicity function 
\begin{align}\label{fR}
f=\frac{B}{A}=\frac{s(s+\mathrm{a}+1)(s-\mathrm{a}-N+1)(s+\mathrm{a}+\alpha+N+1)(s-\mathrm{a}+\beta+1)}{(s+1)(s-\mathrm{a})(s+\mathrm{a}+N)(s-\mathrm{a}-\alpha-N)(s+\mathrm{a}-\beta)}
\end{align}
is a positive and strictly decreasing function of $s \in\mathrm{K}$, and
\begin{align}
\label{DfR1}\frac{\partial f}{\partial \alpha}&=\frac{s(2s+1)(s+\mathrm{a}+1)(s-\mathrm{a}-N+1)(s-\mathrm{a}+\beta+1)}{(s+1)(s-\mathrm{a})(s+\mathrm{a}+N)(s-\mathrm{a}-\alpha-N)^2(s+\mathrm{a}-\beta)}<0,\\[7pt]
\label{DfR2}\frac{\partial f}{\partial \beta}&=\frac{s(2s+1)(s+\mathrm{a}+1)(s-\mathrm{a}-N+1)(s+\mathrm{a}+\alpha+N+1)}{(s+1)(s-\mathrm{a})(s+\mathrm{a}+N)(s-\mathrm{a}-\alpha-N)(s+\mathrm{a}-\beta)^2}>0,
\end{align}
for each $s \in\mathrm{K}$. It is immediate that $\mathrm{S}_{(-1,2\mathrm{a}]}^{(\beta)}\subset \mathrm{K}$. We next claim that $\mathrm{S}_{(2\mathrm{a}, 2\mathrm{a}+1)}^{(\beta)}\subset \mathrm{K}$. Indeed, this is equivalent to prove that $R_n^{(\alpha,\beta)}$ has no zeros on $\big(\mathrm{a}(\mathrm{a}+1), (\beta-\mathrm{a})(\beta-\mathrm{a}+1)\big)$ for $\beta \in (2\mathrm{a}, 2\mathrm{a}+1)$. For $\beta \in (2\mathrm{a}, 2\mathrm{a}+1)$, $R_n^{(\alpha, \beta)}(\mathrm{a}(\mathrm{a}+1))=1$ and 
\begin{align*}
 R_n^{(\alpha, \beta)}\big((\beta-\mathrm{a})(\beta-\mathrm{a}+1)\big)&=\pFq{3}{2}{-n, \alpha+\beta+n+1, 2\mathrm{a}-\beta}{2\mathrm{a}+\alpha+N+1, 1-N}{1}\\[7pt]
 &=\frac{(2\mathrm{a}-\beta+N-n)_n(\alpha+\beta+N+1)_n}{(2\mathrm{a}+\alpha+N+1)_n(N-n)_n}>0,
\end{align*}
the last equality being a consequence of Sheppard's identity (see \cite[Corollary 3.3.4]{AAR99}). Hence $R_n^{(\alpha, \beta)}$ has no zeros on $\big(\mathrm{a}(\mathrm{a}+1), (\beta-\mathrm{a})(\beta-\mathrm{a}+1)\big)$ or has at least two zeros there. Suppose the second of these possibilities is true. From the proof of Lemma \ref{lemma}, we have $x(\mathrm{a}+1)<x(\beta-\mathrm{a})$, which is impossible. Thus $\mathrm{S}_{(2\mathrm{a}, 2\mathrm{a}+1)}^{(\beta)}\subset \mathrm{K}$ as claimed. The same proof actually shows that $\mathrm{S}_{(-1,\infty)}^{(\alpha)}\subset \mathrm{K}$. The result follows from the discrete Stieltjes theorem.
\end{proof}

\begin{obs}
In \cite[Section 9.2]{KLS10}, the Racah polynomials (interchanging $\alpha$ and $\beta$) are defined by
$$
\widetilde{R}_n^{(\alpha, \beta, \delta)}(\widetilde{x}(s))=\pFq{4}{3}{-n, \alpha+\beta+n+1, -s, s+\delta-N+1}{\alpha+\delta+1, \beta+1, 1-N}{1}
$$
$(n=1,\dots, N-1; \mathrm{a}=0, \mathrm{b}=N; \alpha>-1, -1<\beta<\delta-N+1, \delta>N-1)$, where $\widetilde{x}(s)=s(s+\delta-N+1)$.  In \eqref{FRacah}, we can write $R^{(\alpha, \beta, \mathrm{a})}_n$ instead of $R^{(\alpha, \beta)}_n$; although, implicitly, we have the agreement to omit those parameters that are assumed fixed. Clearly, $\widetilde{x}$ is not one of the canonical forms \eqref{grid}. However, $\widetilde{x}$ and $x$ are related by a linear transformation. Therefore, for  $\mathrm{a}=(\delta-N)/2$ fixed,
$$
R^{(\alpha, \beta, (\delta-N)/2)}_n(x(s))=\widetilde{R}_n^{(\alpha, \beta, \delta)}\left(\widetilde{x}\left(s-\frac{\delta-N}{2}\right)\right),
$$
where
$$
\widetilde{x}\left(s-\frac{\delta-N}{2}\right)=x(s)-\frac{(\delta-N+1)^2-1}{4}.
$$
Consequently, Proposition \ref{Racah} remains valid if we replace $R^{(\alpha, \beta)}_n$ by $\widetilde{R}_n^{(\alpha, \beta, \delta)}$.
\end{obs}

The {\em dual Hahn polynomials} (see \cite[p. 236]{NU93}), 
\begin{align}\label{DualHahneq}
y(s)=W_n^{(\alpha)}(X)&=\pFq{3}{2}{-n, \mathrm{a}-s, s+\mathrm{a}+1}{\alpha+1, 1-N}{1}
\end{align}
$(n=1,\dots, N-1; \mathrm{a}>-1/2, \mathrm{b}=\mathrm{a}+N; -1<\alpha<2\mathrm{a}+1)$, satisfy \eqref{DE2} with $A$ and $B$ given by
\begin{align*}
A&=A(s; \alpha)=\frac{(s-\mathrm{a})(s+\mathrm{a}+N)(s+\mathrm{a}-\alpha)}{2s(2s+1)},\\[7pt]
B&=B(s; \alpha)=\frac{(s+\mathrm{a}+1)(-s+\mathrm{a}+N-1)(s-\mathrm{a}+\alpha+1)}{2(s+1)(2s+1)}.
\end{align*}

\begin{proposition}\label{DualHahn}
The zeros of $W_n^{(\alpha)}$ are strictly increasing functions of $\alpha$ on $(-1, 2\mathrm{a}+1)$ if $\mathrm{a}\geq 0$, or else the zeros of $W_n^{(\alpha)}$ are strictly increasing functions of $\alpha$ on $(\mathrm{a}, 2\mathrm{a}+1)$.
\end{proposition} 
\begin{proof}
We sketch the proof only for the case in which $\mathrm{a}\geq 0$. The proof for $-1/2<\mathrm{a}<0$ is similar. Define the interval $\mathrm{K}=\big(\max\big\{\mathrm{a}, \alpha-\mathrm{a}\big\}, \mathrm{a}+N-1\big)$. Note that the hypotheses of the discrete Stieltjes theorem are fulfilled in $\mathrm{K}$. Thus we only need to prove that $\mathrm{S}_{(2\mathrm{a}, 2\mathrm{a}+1)}\subset \mathrm{K}$\footnote{It is immediate that $\mathrm{S}_{(-1, 2\mathrm{a}]} \subset \mathrm{K}$.}.  Note that $W_n^{(\alpha)}(\mathrm{a}(\mathrm{a}+1))=1$ and 
\begin{align*}
 W_n^{(\alpha)}\big((\alpha-\mathrm{a})(\alpha-\mathrm{a}+1)\big)=\pFq{2}{1}{-n, 2\mathrm{a}-\alpha}{1-N}{1}=\frac{(1-N+\alpha-2\mathrm{a})_n}{(1-N)_n}>0,
 \end{align*}
the last equality being a consequence of Chu-Vandermonde's identity (\cite[Corollary 2.2.3]{AAR99}). The rest of the proof runs as in Proposition \ref{Racah}.
\end{proof}

\begin{obs}
In \cite[Section 9.6]{KLS10}, the dual Hahn polynomials are defined by
$$
\widetilde{W}_n^{(\alpha, \beta)}(\widetilde{x}(s))=\pFq{3}{2}{-n, -s, s+\alpha+\beta+1}{\alpha+1, 1-N}{1}
$$
$(n=1,\dots, N-1; \mathrm{a}=0, \mathrm{b}=N; \alpha>-1, \beta>-1\; \text{or}\; \alpha<1-N, \beta<1-N)$, where $\widetilde{x}(s)=s(s+\alpha+\beta+1)$.  In \eqref{DualHahneq}, we can write $W^{(\alpha, \mathrm{a})}_n$ instead of $W^{(\alpha)}_n$. Hence, for $\mathrm{a}=(\alpha+\beta)/2$ fixed,
$$
W^{(\alpha, (\alpha+\beta)/2)}_n(x(s))=\widetilde{W}_n^{(\alpha, \beta)}\left(\widetilde{x}\left(s-\frac{\alpha+\beta}{2}\right)\right),
$$
where
$$
\widetilde{x}\left(s-\frac{\alpha+\beta}{2}\right)=x(s)-\frac{(\alpha+\beta+1)^2-1}{4}.
$$
Consequently, Proposition \ref{DualHahn} remains valid if we replace $W^{(\alpha)}_n$ by $\widetilde{W}_n^{(\alpha, \beta)}$ and assume that $\alpha+\beta$ is constant.
\end{obs}

\subsection{The grid $(\mathrm{III})$}\label{qlinear} 
\subsubsection{Examples of COPRL on $X=x(s)=q^{-s}$ $(0<q<1)$ (q-Meixner, Al-Salam-Carlitz, q-Charlier, q-Hahn, q-Krawtchouk, affine q-Krawtchouk, and quantum q-Krawtchouk) and some related cases (q-Charlier and big q-Laguerre).}
\vspace{.2cm}
The {\em q-Meixner polynomials} (see \cite[Section 14.13]{KLS10}), 
$$
y(s)=M_n^{(\alpha, \beta)}(X; q)=\pPq{2}{1}{q^{-n}, X}{\beta q}{q, -\frac{q^{n+1}}{\alpha}}
$$
$(n=1,2,\dots; \mathrm{a}=0, \mathrm{b}=\infty; \alpha>0, 0\leq\beta<q^{-1})$, satisfy \eqref{DE2} with $A$ and $B$ given by
\begin{align*}
A=A(s; \alpha, \beta)=(1-q^s)(1+\alpha \beta q^s), \quad B=B(s; \alpha, \beta)=\alpha q^s(1-\beta q^{s+1}).
\end{align*}

\begin{proposition}
The zeros of $M_n^{(\alpha, \beta)}(\cdot; q)$ are strictly increasing functions of $\alpha$ on $(0, \infty)$ and strictly decreasing functions of $\beta$ on $[0,q^{-1})$.
\end{proposition}
\begin{proof}
The monotonicity function
$$
f=\frac{B}{A}=\frac{\alpha q^s(1-\beta q^{s+1})}{(1-q^s)(1+\alpha \beta q^s)}
$$
is a positive and strictly decreasing function of $s \in (0, \infty)$, and 
\begin{align*}
\frac{\partial f}{\partial \alpha}&=\frac{q^s(1-\beta q^{s+1})}{(1-q^s)(1+\alpha \beta q^s)^2}>0,\\[7pt]
\frac{\partial f}{\partial \beta}&=-\frac{\alpha q^{2s} (\alpha+q)}{(1-q^s)(1+\alpha \beta q^s)^2}<0,
\end{align*}
for each $s \in(0, \infty)$. The result follows from the discrete Stieltjes theorem.
\end{proof}
\begin{obs}
The q-Charlier polynomials (see \cite[Section 14.23]{KLS10}) are given by $C_n^{(\alpha)}(\cdot; q)=M_n^{(\alpha, 0)}(\cdot; q)$.
\end{obs}

The {\em second family of Al-Salam-Carlitz polynomials} (see \cite[Section 14.25]{KLS10}),
$$
y(s)=V_n^{(\alpha)}(X; q)=(-\alpha)^n q^{-\binom{n}{2}} \pPq{2}{0}{q^{-n}, X}{-}{q, \frac{q^n}{\alpha}}
$$ 
$(n=1,2,\dots; \mathrm{a}=0, \mathrm{b}=\infty; 0<\alpha<q^{-1})$, satisfy \eqref{DE2} with $A$ and $B$ given by
\begin{align*}
A=A(s; \alpha)=(1-q^{-s})(\alpha-q^{-s}),\quad B=B(s; \alpha)=\alpha q.
\end{align*}

\begin{proposition}\label{aSC}
The zeros of $V_n^{(\alpha)}(\cdot; q)$ are strictly increasing functions of $\alpha$ on $(0, q^{-1})$.
\end{proposition} 
\begin{proof}
Define the interval $\mathrm{K}=\big(\max\{-\log_q \alpha, 0\}, \infty\big)$. Note that the hypotheses of the discrete Stieltjes theorem are fulfilled in $\mathrm{K}$. Thus we only need to prove that $\mathrm{S}_{(1, q^{-1})}\subset \mathrm{K}$\footnote{It is immediate that $\mathrm{S}_{(0, 1]} \subset \mathrm{K}$.}.  The rest of the proof runs as in Proposition \ref{Racah}; although given the simplicity of this case, we do not need to use q-hypergeometric identities.
\end{proof}

\begin{obs}
The first family of Al-Salam-Carlitz polynomials is given by $U_n^{(\alpha)}$ $(\cdot; q^{-1})=V_n^{(\alpha)}(\cdot; q)$.
\end{obs}

The {\em q-Hahn polynomials} (see \cite[Section 14.6]{KLS10}), 
$$
y(s)=H_n^{(\alpha, \beta)}(X; q)=\pPq{3}{2}{q^{-n}, \alpha \beta q^{n+1}, X}{\alpha q, q^{1-N}}{q, q}
$$
$(n=1,\dots N-1; \mathrm{a}=0, \mathrm{b}=N; 0<\alpha<q^{-1}, 0<\beta<q^{-1})$, satisfy \eqref{DE2} with $A$ and $B$ given by
\begin{align*}
A&=A(s; \alpha, \beta)=\alpha q (1-q^s)(\beta-q^{s-N}),\\[7pt]
B&=B(s; \alpha)=(1-q^{s-N+1})(1-\alpha q^{s+1}).
\end{align*}
\begin{proposition}
The zeros of $H_n^{(\alpha, \beta)}(\cdot; q)$ are strictly decreasing functions of $\alpha$ on $(0, q^{-1})$ and strictly increasing functions of $\beta$ on $(0,q^{-1})$.
\end{proposition} 

The {\em q-Krawtchouk polynomials} (see \cite[Section 14.15]{KLS10}), 
$$
y(s)=K_n^{(\alpha)}(X; q)=\pPq{3}{2}{q^{-n}, -\alpha q^{n}, X}{0, q^{1-N}}{q, q}
$$
$(n=1,\dots, N-1; \mathrm{a}=0, \mathrm{b}=N; \alpha>0)$, satisfy \eqref{DE2} with $A$ and $B$ given by
\begin{align*}
A=A(s; \alpha)=\alpha (q^s-1),\quad B=B(s)=1-q^{s-N+1}.
\end{align*}
\begin{proposition}
The zeros of $K_n^{(\alpha)}(\cdot; q)$ are strictly decreasing functions of $\alpha$ on $(0, \infty)$.
\end{proposition} 

The {\em affine q-Krawtchouk polynomials} (see \cite[Section 14.16]{KLS10}), 
$$
y(s)=\widehat{K}_n^{(\alpha)}(X; q)=\pPq{3}{2}{q^{-n},0, X}{\alpha q, q^{1-N}}{q, q}
$$
$(n=1,\dots, N-1; \mathrm{a}=0, \mathrm{b}=N; 0<\alpha<q^{-1})$, satisfy \eqref{DE2} with $A$ and $B$ given by
\begin{align*}
A=A(s; \alpha)=\alpha q^{s-N+1} (q^s-1),\quad B=B(s; \alpha)=(1-q^{s-N+1})(1-\alpha q^{s+1}).
\end{align*}
\begin{proposition}
The zeros of $\widehat{K}_n^{(\alpha)}(\cdot; q)$ are strictly decreasing functions of $\alpha$ on $(0, q^{-1})$.
\end{proposition} 
\begin{obs}\label{bqL1}
The big q-Laguerre polynomials (see \cite[Section 14.16]{KLS10}) are given by
$$
L_n^{(\alpha, \beta)}(X; q)=\pPq{3}{2}{q^{-n}, 0, X}{\alpha q, \beta q}{q, q}
$$
$(n=1,2,\dots ; 0<\alpha<q^{-1}, \beta<0)$. In particular, $\widehat{K}_n^{(\alpha)}(\cdot; q)=L_n^{(\alpha, q^{-N})}(X; q)$ (see also Remark \ref{bqL2}).
\end{obs}

The {\em quantum q-Krawtchouk polynomials} (see \cite[Section 14.14]{KLS10}), 
$$
y(s)=\widetilde{K}_n^{(\alpha)}(X; q)=\frac{(q^{-N},q)_n}{\alpha^n q^{n^2}} \pPq{2}{1}{q^{-n}, X}{q^{1-N}}{q, \alpha q^{n+1}}
$$
$(n=1,\dots, N-1; \mathrm{a}=0, \mathrm{b}=N; \alpha>q^{1-N})$, satisfy \eqref{DE2} with $A$ and $B$ given by
\begin{align*}
A=A(s; \alpha)=(1-q^s)(\alpha-q^{s-N}),\quad B=B(s)=-q^s(1-q^{s-N+1}).
\end{align*}
\begin{proposition}
The zeros of $\widetilde{K}_n^{(\alpha)}(\cdot; q)$ are strictly decreasing functions of $\alpha$ on $(q^{1-N}, \infty)$.
\end{proposition} 
\begin{proof}
Define the interval $\mathrm{K}=\big(\max\{0, \log_q \alpha+N\}, N-1\big)$. Note that the hypotheses of the discrete Stieltjes theorem are fulfilled in $\mathrm{K}$. Thus we only need to prove that $\mathrm{S}_{(q^{1-N}, q^{-N})}\subset \mathrm{K}$\footnote{It is immediate that $\mathrm{S}_{[q^{-N}, \infty)}\subset \mathrm{K}$.}.  The rest of the proof runs as in Proposition \ref{Racah}; although given the simplicity of this case, we do not need to use q-hypergeometric identities.
\end{proof}

\subsubsection{Examples of COPRL on $X=x(s)=q^{s}$ $(0<q<1)$ (q-Bessel, little q-Jacobi, and little q-Laguerre/Wall) and some related cases (big q-Jacobi, big q-Laguerre, and q-Laguerre).}
\vspace{.2cm}
The {\em q-Bessel} (see \cite[Section 14.22]{KLS10}), 
$$
y(s)=B_n^{(\alpha)}(X; q)=\pPq{2}{1}{q^{-n}, -\alpha q^{n}}{0}{q, q X}
$$
$(n=1, 2,\dots; \mathrm{a}=0, \mathrm{b}=\infty; \alpha>0)$, satisfy \eqref{DE2} with $A$ and $B$ given by
\begin{align*}
A=A(s)=q^{s}-1,\quad B=B(s; \alpha)=\alpha.
\end{align*}
\begin{proposition}
The zeros of $B_n^{(\alpha)}(\cdot; q)$ are strictly decreasing functions of $\alpha$ on $(0, \infty)$.
\end{proposition} 

The {\em little q-Jacobi polynomials} (see \cite[Section 14.12]{KLS10}), 
$$
y(s)=P_n^{(\alpha, \beta)}(X; q)=\pPq{2}{1}{q^{-n}, \alpha \beta q^{n+1}}{\alpha q}{q, q X}
$$
$(n=1, 2,\dots; \mathrm{a}=0, \mathrm{b}=\infty; 0<\alpha<q^{-1}, \beta<q^{-1})$, satisfy \eqref{DE2} with $A$ and $B$ given by
\begin{align*}
A=A(s)=q^{-s} (q^{s}-1),\quad B=B(s; \alpha, \beta)=\alpha q^{-s} (\beta q^{s+1}-1).
\end{align*}
\begin{proposition}
The zeros of $P_n^{(\alpha, \beta)}(\cdot; q)$ are strictly decreasing functions of $\alpha$ on $(0, q^{-1})$ and strictly increasing functions of $\beta$ on $(-\infty,q^{-1})$.
\end{proposition} 

\begin{coro}
The zeros of the special case of big q-Jacobi polynomials (see \cite[Section 14.5]{KLS10}), 
$$
P_n^{(\alpha, \beta, 0)}(X; q)=\frac{(\beta q; q)_n}{(\alpha q; q)_n} (-1)^n \alpha^n q^{n+\binom{n}{2}} P_n^{(\beta, \alpha)}(\alpha^{-1} q^{-1} X; q)
$$
 $(n=1, 2,\dots; 0<\alpha<q^{-1}, 0<\beta<q^{-1})$, are strictly increasing functions of $\alpha$ on $(0, q^{-1})$ and strictly decreasing functions of $\beta$ on $(0,q^{-1})$.
\end{coro}

\begin{obs}\label{bqL2}
$P_n^{(\alpha, 0, 0)}(\cdot; q)=L_n^{(\alpha, 0)}(\cdot; q)$ (see also Remark \ref{bqL1}).
\end{obs}

The {\em little q-Laguerre/Wall polynomials} (see \cite[Section 14.20]{KLS10}), 
$$
y(s)=L_n^{(\alpha)}(X; q)=\pPq{2}{1}{q^{-n}, 0}{\alpha q}{q, q X}
$$
$(n=1, 2,\dots; \mathrm{a}=0, \mathrm{b}=\infty; 0<\alpha<q^{-1})$, satisfy \eqref{DE2} with $A$ and $B$ given by
\begin{align*}
A=A(s)=q^{s}-1,\quad B=B(s; \alpha)=\alpha q^{-s}.
\end{align*}
\begin{proposition}
The zeros of $L_n^{(\alpha)}(\cdot; q)$ are strictly decreasing functions of $\alpha$ on $(0, q^{-1})$.
\end{proposition} 

\begin{coro}
The zeros of the q-Laguerre polynomials (see \cite[Section 14.21]{KLS10}), 
$$
\widehat{L}_n^{(\alpha)}(\cdot; q)=q^{-\alpha n}\frac{(q^{\alpha+1}; q)_n}{(q; q)_n}L_n^{(q^\alpha)}(\cdot; q)
$$
 $(n=1, 2,\dots; \alpha>-1)$, are strictly decreasing functions of $\alpha$ on $(-1, \infty)$.
\end{coro}

\subsection{The grid $(\mathrm{VI})$}\label{latticeII}
\subsubsection*{Examples of COPRL on $X=x(s)=(q^{s}+q^{-s})/2$ $(0<q<1)$ (q-Racah and dual q-Hahn polynomials).}

The {\em q-Racah polynomials} (see \cite[p. 239]{NU93}), 
\begin{align}\label{qFRacah}
y(s)=R_n^{(\alpha, \beta)}(X; q)&=\pPq{4}{3}{q^{-n}, q^{\alpha+\beta+n+1}, q^{\mathrm{a}-s},q^{s+\mathrm{a}}}{q^{2\mathrm{a}+\alpha+N}, q^{\beta+1}, q^{1-N}}{q,q}
\end{align}
$(n=1,\dots, N-1; \mathrm{a}>0, \mathrm{b}=\mathrm{a}+N; \alpha>-1, -1<\beta<2\mathrm{a})$, satisfy \eqref{DE2} with $A$ and $B$ given by
\begin{align*}
A&=A(s; \alpha, \beta)\\[7pt]
&=-\frac{4\,q^{\alpha+\beta+5/2}(q^{s-\mathrm{a}}-1)(q^{s+\mathrm{a}+N-1}-1)(q^{s-\mathrm{a}-\alpha-N}-1)(q^{s+\mathrm{a}-\beta-1}-1)}{(q-1)^2(q^{2s}-1)(q^{2s-1}-1)},\\[7pt]
B&=B(s; \alpha, \beta)\\[7pt]
&=-\frac{4\, q^{3/2}(q^{s+\mathrm{a}}-1)(q^{s-\mathrm{a}-N+1}-1)(q^{s+\mathrm{a}+\alpha+N}-1)(q^{s-\mathrm{a}+\beta+1}-1)}{(q-1)^2(q^{2s}-1)(q^{2s+1}-1)}.
\end{align*}

\begin{proposition}\label{qRacah}
The zeros of $R_n^{(\alpha,\beta)}(\cdot; q)$ are strictly decreasing functions of $\alpha$ on $(-1, \infty)$ and strictly increasing function of $\beta$ on $(-1, 2\mathrm{a})$ if $\mathrm{a}\geq 1/2$, or else the zeros of $R_n^{(\alpha,\beta)}(\cdot; q)$ are strictly decreasing functions of $\alpha$ on $(-1, \infty)$ for each $\beta\in (\mathrm{a}-1/2, 2\mathrm{a})$ and strictly increasing function of $\beta$ on $(\mathrm{a}-1/2, 2\mathrm{a})$.
\end{proposition} 
\begin{proof}
We give the proof only for the case in which $\mathrm{a}\geq 1/2$. The proof for $0<\mathrm{a}<1/2$ is similar. Define the interval $\mathrm{K}=\big(\max\big\{\mathrm{a}, \beta-\mathrm{a}+1\big\}, \mathrm{a}+N-1\big)$. 
The monotonicity function 
\begin{align*}
f&=\frac{B}{A}\\[7pt]
&=\frac{(q^{s+\mathrm{a}}-1)(q^{s-\mathrm{a}-N+1}-1)(q^{s+\mathrm{a}+\alpha+N}-1)(q^{s-\mathrm{a}+\beta+1}-1)}{q^{\alpha+\beta+1}(q^{s-\mathrm{a}}-1)(q^{s+\mathrm{a}+N-1}-1)(q^{s-\mathrm{a}-\alpha-N}-1)(q^{s+\mathrm{a}-\beta-1}-1)}\frac{q^{2s-1}-1}{q^{2s+1}-1}
\end{align*}
is a positive and strictly decreasing function of $s \in\mathrm{K}$, and
\begin{align*}
&0>\frac{\partial f}{\partial \alpha}=\\[7pt]
&\frac{\log q\,(q^{2s}-1)(q^{2s-1}-1)(q^{s+\mathrm{a}}-1)(q^{s-\mathrm{a}-N+1}-1)(q^{s-\mathrm{a}+\beta+1})}{q^{\alpha+\beta+1/}(q^{2s+1}-1)(q^{s-\mathrm{a}}-1)(q^{s+\mathrm{a}+N-1}-1)(q^{s-\mathrm{a}-\alpha-N}-1)^2(q^{s+\mathrm{a}-\beta-1}-1)},\\[7pt]
&0<\frac{\partial f}{\partial \beta}=\\[7pt]
&\frac{\log q\,(q^{2s}-1)(q^{2s-1}-1)(q^{s+\mathrm{a}}-1)(q^{s-\mathrm{a}-N+1}-1)(q^{s+\mathrm{a}+\alpha+N})}{q^{\alpha+\beta+1}(q^{2s+1}-1)(q^{s-\mathrm{a}}-1)(q^{s+\mathrm{a}+N-1}-1)(q^{s-\mathrm{a}-\alpha-N}-1)(q^{s+\mathrm{a}-\beta-1}-1)^2},
\end{align*}
for each $s \in\mathrm{K}$. Thus we only need to prove that $\mathrm{S}^{(\beta)}_{(2\mathrm{a}-1, 2\mathrm{a})}\subset \mathrm{K}$\footnote{It is immediate that $\mathrm{S}^{(\beta)}_{(-1, 2\mathrm{a}-1]}\subset \mathrm{K}$.}. Indeed, this is equivalent to prove that $R_n^{(\alpha,\beta)}(\cdot; q)$ has no zeros on $\big((q^{\mathrm{a}}+q^{-\mathrm{a}})/2, (q^{\beta-\mathrm{a}+1}+q^{\mathrm{a}-\beta-1})/2\big)$ for $\beta \in (2\mathrm{a}-1, 2\mathrm{a})$. For $\beta \in (2\mathrm{a}-1, 2\mathrm{a})$, $R_n^{(\alpha, \beta)}\big((q^{\mathrm{a}}+q^{-\mathrm{a}})/2; q\big)=1$ and 
\begin{align*}
 R_n^{(\alpha, \beta)}\big((q^{\beta-\mathrm{a}+1}+q^{\mathrm{a}-\beta-1})/2; q\big)&=\pPq{3}{2}{q^{-n}, q^{\alpha+\beta+n+1}, q^{2\mathrm{a}-\beta-1}}{q^{2\mathrm{a}+\alpha+N}, q^{1-N}}{q,q}\\[7pt]
 &=\frac{(q^{2\mathrm{a}-\beta+N-n-1}; q)_n (q^{\alpha+\beta+N+1}; q)_n}{(q^{2\mathrm{a}+\alpha+N} ;q)_n (q^{N-n}; q)_n}>0,
\end{align*}
the last equality being a consequence of q-Pfaff-Saalsch\"utz's identity (see \cite[(10.10.3)]{AAR99}).  We thus get, as in Proposition \ref{Racah}, $\mathrm{S}^{(\beta)}_{(-1, 2\mathrm{a})}\subset \mathrm{K}$. The same proof actually shows that $\mathrm{S}_{(-1,\infty)}^{(\alpha)}\subset \mathrm{K}$. The result follows from the discrete Stieltjes theorem.
\end{proof}

\begin{obs}
In \cite[Section 14.2]{KLS10}, the q-Racah polynomials (replacing $\alpha$ by $q^{\beta}$ and $\beta$ by $q^{\alpha}$) are defined by
$$
\widetilde{R}_n^{(\alpha, \beta, \delta)}(\widetilde{x}(s); q)=\pPq{4}{3}{q^{-n}, q^{\alpha+\beta+n+1}, q^{-s},q^{s-N+1}}{\delta q^{\alpha+1}, q^{\beta+1}, q^{1-N}}{q,q}
$$
$(n=1,\dots, N-1; \mathrm{a}=0, \mathrm{b}=N; \alpha>-1, -1<\beta<1-N+\log_q \delta, 0<\delta<q^{N-1})$, where $\widetilde{x}(s)=\delta q^{s-N+1}+q^{-s}$.  In \eqref{qFRacah}, we can write $R^{(\alpha, \beta, \mathrm{a})}_n(\cdot; q)$ instead of $R^{(\alpha, \beta)}_n(\cdot; q)$. Hence, for  $\mathrm{a}=(1-N+\log_q \delta)/2$ fixed,
$$
R^{(\alpha, \beta, (1-N+\log_q \delta)/2)}_n(x(s); q)=\widetilde{R}_n^{(\alpha, \beta, \delta)}\left(\widetilde{x}\left(s-(1-N+\log_q \delta)/2\right); q\right),
$$
where
$$
\widetilde{x}\left(s-(1-N+\log_q \delta)/2\right)=2 q^{(1-N+\log_q \delta)/2} x(s).
$$
Consequently, Proposition \ref{qRacah} remains valid if we replace $R^{(\alpha, \beta)}_n(\cdot; q)$ by $\widetilde{R}_n^{(\alpha, \beta, \delta)}$ $(\cdot; q)$.\end{obs}

The {\em dual q-Hahn polynomials} (see \cite[p. 239]{NU93}), 
\begin{align}\label{DqH}
y(s)=W_n^{(\alpha)}(X; q)&=\pPq{3}{2}{q^{-n}, q^{\mathrm{a}-s},q^{s+\mathrm{a}}}{q^{\alpha+1}, q^{1-N}}{q,q}
\end{align}
$(n=1,\dots, N-1; \mathrm{a}>0, \mathrm{b}=\mathrm{a}+N; -1<\alpha<2\mathrm{a})$, satisfy \eqref{DE2} with $A$ and $B$ given by
\begin{align*}
A&=A(s; \alpha)=-\frac{4 q^{-\mathrm{a}+\alpha-N+5/2}(q^{s-\mathrm{a}}-1)(q^{s+\mathrm{a}+N-1}-1)(q^{s+\mathrm{a}-\alpha-1}-1)}{(q-1)^2(q^{2s}-1)(q^{2s-1}-1)},\\[7pt]
B&=B(s; \alpha)=-\frac{4 q^{s+3/2}(q^{s+\mathrm{a}}-1)(q^{s-\mathrm{a}-N+1}-1)(q^{s-\mathrm{a}+\alpha+1}-1)}{(q-1)^2(q^{2s}-1)(q^{2s+1}-1)}.
\end{align*}

\begin{proposition}\label{DqHahn}
The zeros of $W_n^{(\alpha)}(\cdot; q)$ are strictly increasing functions of $\alpha$ on $(-1, 2\mathrm{a})$ if $\mathrm{a}\geq 1/2$, or else the zeros of $W_n^{(\alpha)}(\cdot; q)$ are strictly increasing functions of $\alpha$ on $(\mathrm{a}-1/2, 2\mathrm{a})$.
\end{proposition} 
\begin{proof}
We sketch the proof only for the case in which $\mathrm{a}\geq 1/2$. The proof for $0<\mathrm{a}<1/2$ is similar. Define the interval $\mathrm{K}=\big(\max\big\{\mathrm{a}, \alpha-\mathrm{a}+1\big\}, \mathrm{a}+N-1\big)$. Note that the hypotheses of the discrete Stieltjes theorem are fulfilled in $\mathrm{K}$. Thus we only need to prove that $\mathrm{S}_{(2\mathrm{a}-1, 2\mathrm{a})}\subset \mathrm{K}$\footnote{It is immediate that $\mathrm{S}_{(-1, 2\mathrm{a}-1]} \subset \mathrm{K}$.}.  Note that $W_n^{(\alpha)}\big((q^{\mathrm{a}}+q^{-\mathrm{a}})/2; q\big)=1$ and 
\begin{align*}
 W_n^{(\alpha)}\big((q^{\alpha-\mathrm{a}+1}+q^{\mathrm{a}-\alpha-1})/2; q\big)&=\pPq{2}{1}{q^{-n}, q^{2\mathrm{a}-\alpha-1}}{q^{1-N}}{q, q}\\[7pt]
 &=q^{n(2\mathrm{a}-\alpha-1)}\frac{(q^{-2\mathrm{a}+\alpha-N-2}; q)_n}{(q^{1-N}; q)_n}>0,
 \end{align*}
the last equality being a consequence of q-Chu-Vandermonde's identity (\cite[(1.11.5)]{KLS10}). The rest of the proof runs as in Proposition \ref{qRacah}.
\end{proof}

\begin{obs}
In \cite[Section 14.7]{KLS10}, the dual q-Hahn polynomials are defined by
$$
\widetilde{W}_n^{(\alpha, \beta)}(\widetilde{x}(s); q)=\pPq{3}{2}{q^{-n}, q^{\mathrm{a}-s},q^{s+\mathrm{a}}}{q^{\alpha+1}, q^{1-N}}{q,q}
$$
$(n=1,\dots, N-1; \mathrm{a}=0, \mathrm{b}=N; \alpha>-1, \beta>-1\; \text{or}\; \alpha<-N, \beta<-N)$, where $\widetilde{x}(s)=\delta q^{s+\alpha+\beta+1}+q^{-s}$.  In \eqref{DqH}, we can write $W^{(\alpha, \mathrm{a})}_n(\cdot; q)$ instead of $W^{(\alpha)}_n(\cdot; q)$. Hence, for  $\mathrm{a}=(\alpha+\beta+1)/2$ fixed,
$$
W^{(\alpha, (\alpha+\beta+1)/2)}_n(x(s); q)=\widetilde{W}_n^{(\alpha, \beta)}\left(\widetilde{x}\left(s-(\alpha+\beta+1)/2\right); q\right),
$$
where
$$
\widetilde{x}\left(s-(\alpha+\beta+1)/2\right)=2 q^{(\alpha+\beta+1)/2} x(s).
$$
Consequently, Proposition \ref{DqHahn} remains valid if we replace $W^{(\alpha)}_n(\cdot; q)$ by $\widetilde{W}_n^{(\alpha, \beta)}$ $(\cdot; q)$ and assume that $\alpha+\beta$ is constant.
\end{obs}

\section*{Acknowledgements}
The authors thank the Warsaw University Library for kindly sending them the book \cite{NU78} and the Keldysh Institute of Applied Mathematics for making \cite{NU83} available at \url{https://keldysh.ru/papers/1983/prep1983_17.pdf} after they request. KC is supported by the Centre for Mathematics of the University of Coimbra - UIDB/00324/2020, funded by the Portuguese Government through FCT/ MCTES. FRR and AS are supported by the Funda\c{c}\~ao de Amparo \`a Pesquisa do Estado de Minas Gerais (FAPEMIG) Demanda Universal under the grant APQ-03782-18, Conselho Nacional de Desenvolvimento Cient\'\i fico e Tecnol\'ogico (CNPq), and Coordena\c{c}\~ao de Aperfei\c{c}oamento de Pessoal de N\'\i vel Superior (CAPES).
\bibliographystyle{plain}
 \bibliography{bib}

\begin{thebibliography}{10}

\bibitem{AAR99}
G.~E. Andrews, R.~Askey, and R.~Roy.
\newblock {\em Special {F}unctions}, volume~71 of {\em Encyclopedia of
  Mathematics and its Applications}.
\newblock Cambridge University Press, Cambridge, 1999.

\bibitem{AW85}
R.~Askey and J.~Wilson.
\newblock Some basic hypergeometric orthogonal polynomials that generalize
  jacobi polynomials.
\newblock {\em Mem. Amer. Math. Soc.}, 54(319), 1985.

\bibitem{ARS95}
N.~M. Atakishiyev, M.~Rahman, and S.~K. Suslov.
\newblock On classical orthogonal polynomials.
\newblock {\em Constr. Approx.}, 11:181--226, 1995.

\bibitem{I05}
M.~E.~H. Ismail.
\newblock {\em Classical and quantum orthogonal polynomials in one variable},
  volume~98 of {\em Encyclopedia of Mathematics and Its Applications}.
\newblock Cambridge University Press, Cambridge, 2005.

\bibitem{KM60}
S.~Karlin and J.~L. McGregor.
\newblock The {H}ahn polynomials, formulas and an application.
\newblock Technical Report~2, Applied {M}athematics and {S}tatistics
  {L}aboratorires, {S}tanford {U}niversity, Stanford, CA, April 1960.

\bibitem{KLS10}
R.~Koekoek, P.~A. Lesky, and R.~F. Swarttouw.
\newblock {\em Hypergeometric {O}rthogonal {P}olynomials and {T}heir
  q-{Analogues}}.
\newblock Springer {M}onographs in {M}athematics, 2010.

\bibitem{L67}
R.~J. Levit.
\newblock The zeros of the {H}ahn polynomials.
\newblock {\em SIAM Rev.}, 9:191--203, 1967.

\bibitem{M86}
A.~Markoff.
\newblock Sur les racines de certaines \'equations (second note).
\newblock {\em Math. Ann.}, 27:177--182, 1886.

\bibitem{M91}
P.~Maroni.
\newblock Une th\'eorie alg\'ebrique des polyn\^{o}mes orthogonaux. application
  aux polyn\^{o}mes orthogonaux semi-classiques.
\newblock In C.~Brezinski, L.~Gori, and A.~Ronveaux, editors, {\em Orthogonal
  Polynomials and Their Applications}, volume~9 of {\em IMACS Annals Comput
  Appl Math.}, pages 95--130, 1991.

\bibitem{Ma66}
Ch.~K. Mesztenyi.
\newblock Orthogonal polynomials on a finite set.
\newblock Technical Report TR-66-30, University of Maryland, 1966.

\bibitem{NSU85}
A.~F. Nikiforov, S.~K. Suslov, and V.~B. Uvarov.
\newblock {\em {\scshape Классические ортогональные
  полиномы дискретной переменной} ({R}ussian)
  [{C}lassical orthogonal polynomials of a discrete variable]}.
\newblock ``Nauka'', Moscow, 1985.

\bibitem{NSU}
A.~F. Nikiforov, S.~K. Suslov, and V.~B. Uvarov.
\newblock {\em Classical orthogonal polynomials of a discrete variable.
  {T}ranslated from the {R}ussian}.
\newblock Springer Series in Computational Physics. Springer-Verlag, 1991.

\bibitem{NU78}
A.~F. Nikiforov and V.~B. Uvarov.
\newblock {\em {\scshape Специальные функции
  математической физики.} ({R}ussian) [{S}pecial
  {F}unctions of {M}athematical {P}hysics] With a preface by A. A. Samarskii.}
\newblock ``Nauka'', Moscow, 1978.

\bibitem{NU83}
A.~F. Nikiforov and V.~B. Uvarov.
\newblock {\scshape Классические ортогональные
  полиномы дискретной переменной на
  неравномерных сетках} (russian) [classical orthogonal
  polynomials of a discrete variable].
\newblock Preprint~17, Keldysh Inst. Appl. Math., 1983.

\bibitem{NU84}
A.~F. Nikiforov and V.~B. Uvarov.
\newblock {\em {\scshape Специальные функции
  математической физики.} ({R}ussian) [{S}pecial
  {F}unctions of {M}athematical {P}hysics] With a preface by A. A. Samarskii.}
\newblock ``Nauka'', Moscow, second edition, 1984.

\bibitem{NU93}
A.~F. Nikiforov and V.~B. Uvarov.
\newblock Polynomial solutions of hypergeometric type difference equations and
  their classification.
\newblock {\em Integral Transforms Spec. Funct.}, 1:223--249, 1993.

\bibitem{P18}
J.~Petronilho.
\newblock Orthogonal polynomials and special functions.
\newblock Class notes for a course given in the UC$|$UP Joint PhD Program in
  Mathematics, University of Coimbra, 2018.

\bibitem{R85}
G.-C. Rota and D.~Sharp.
\newblock {M}athematics, {P}hilosophy, and {A}rtificial {I}ntelligence... a
  dialogue with {G}ian-{C}arlo {R}ota and {D}avid {S}harp.
\newblock {\em Los {A}lamos {S}cience}, Spring/Summer(12):92--104, 1985.

\bibitem{Si15}
B.~Simon.
\newblock {\em Basic {C}omplex {A}nalysis. {A} {C}omprehensive {C}ourse in
  {A}nalysis, {P}art {2A}}.
\newblock American Mathematical Society, Providence, RI, 2015.

\bibitem{Si15O}
B.~Simon.
\newblock {\em Opetator {Theory}. {A} {C}omprehensive {C}ourse in {A}nalysis,
  {P}art {4}}.
\newblock American Mathematical Society, Providence, RI, 2015.

\bibitem{S87}
T.~J. Stieltjes.
\newblock Sur les racines de l'equation ${X}_n=0$.
\newblock {\em Acta Math.}, 9:385--400, 1887.

\bibitem{SH}
T.~J. Stieltjes and Ch. Hermite.
\newblock {\em Correspondance d'Hermite et de Stieltjes. {V}ol. I}.
\newblock Gauthier-Villars, Paris, 1905.

\bibitem{S75}
G.~Szeg\H{o}.
\newblock {\em Orthogonal polynomials}, volume~23.
\newblock Amer. Math. Soc. Coll. Publ., Amer. Math. Soc., Providence, R. I.,
  revised edition, 1959.

\bibitem{V62}
R.~S. Varga.
\newblock {\em Matrix iterative analysis}.
\newblock Prentice-Hall, Inc., Englewood Cliffs, N.J., 1962.

\bibitem{VK91}
N.~Ja. Vilenkin and U.~A. Klimyk.
\newblock {\em Representation of {L}ie groups and special functions. {V}ol. 1.
  {S}implest {L}ie groups, special functions and integral transforms.
  {T}ranslated from the {R}ussian by {V}. {A}. {G}roza and {A}. {A}. {G}roza.},
  volume~72 of {\em Mathematics and its Applications (Soviet Series)}.
\newblock Kluwer Academic Publishers Group, Dordrecht, 1991.

\end{thebibliography}
  \end{document}